\documentclass[12pt]{amsart}
\usepackage{amsmath, amssymb, enumitem, float, frcursive, mathrsfs}
\usepackage[headings]{fullpage}

\numberwithin{equation}{section}

\theoremstyle{plain}

\newtheorem{theorem}[subsubsection]{Theorem}
\newtheorem{lemma}[subsubsection]{Lemma}
\newtheorem{proposition}[subsubsection]{{Proposition}}

\newtheorem{question}[subsubsection]{{Question}}

\theoremstyle{definition}
\newtheorem{definition}[subsubsection]{{Definition}}
\newtheorem{example}[subsubsection]{{Example}}

\theoremstyle{remark}
\newtheorem{remark}[subsubsection]{{Remark}}

\newtheoremstyle{RepTheorem} 
	{6pt}
	{\topsep}
	{\itshape}
	{}
	{\bfseries}
	{.}
	{ .5em}
	{\thmname{#1}\thmnote{ \bfseries #3}}
\theoremstyle{RepTheorem}
\newtheorem{reptheorem}[subsubsection]{Theorem}

\def \Dri {\textnormal{\textbf{\footnotesize{\cursive{Dri}}}}\hspace{2pt}}

\def\Z {{\mathbb Z}}
\def\F {{\mathbb F}}
\def\Q {{\mathbb Q}}
\def\R {{\mathbb R}}
\def\C {{\mathbb C}}
\def\A {{\mathbb A}}

\def\P {{\mathbb P}}
\def\G {{\mathbb G}}

\def\O {{\mathcal O}}

\def\mcA {{\mathcal A}}
\def\mcB {{\mathcal B}}
\def\mcC {{\mathcal C}}
\def\mcD {{\mathcal D}}

\def\mcF {{\mathcal F}}

\def\mcI {{\mathcal I}}
\def\mcJ {{\mathcal J}}
\def\mcL {{\mathcal L}}

\def\mcP {{\mathcal P}}

\def \msL {{\mathscr L}}

\def \a {{\mathfrak a}}

\def \d {{\mathfrak d}}

\def \p {{\mathfrak p}}

\def \mfD {{\mathfrak D}}

\def \w {{\textnormal{\textbf{w}}}}

\newcommand{\aff}{\textnormal{aff}}
\newcommand{\Div}{\textnormal{div}}
\newcommand{\DIV}{\textnormal{Div}}

\newcommand{\Ht}{\textnormal{ht}}

\newcommand{\Pic}{\textnormal{Pic}}

\newcommand{\proj}{\textnormal{proj}}

\newcommand{\Val}{\textnormal{Val}}
\newcommand{\ord}{\textnormal{ord}}

\renewcommand{\Re}{\textnormal{Re}}

\usepackage{microtype}

\begin{document}

\title{Counting Drinfeld Modules with Prescribed Local Conditions}

\author{Tristan Phillips}

\email{tristanphillips72@gmail.com}



\begin{abstract}
We give asymptotics for the number of isomorphism classes of Drinfeld $\F_q[T]$-modules over $\F_q(T)$ of a given height, which satisfy prescribed sets of local conditions. This is done by relating our problem to a problem about counting points on weighted projective stacks. Our results for counting points of bounded height on weighted projective stacks over global function fields may be of independent interest.
\end{abstract}

\maketitle

\section{Introduction}

Much of the arithmetic of an elliptic curve over a global field is reflected in its reduction types at various primes. It is therefore natural to ask, \emph{How many elliptic curves over a global field satisfy a prescribed set of local conditions?} This question has been recently studied by Cho and Jeong \cite[Theorem 1.4]{CJ23}, who studied the case of elliptic curves over the rational numbers in which finitely many local conditions are imposed (excluding conditions at the primes $2$ and $3$). In \cite{CS23}, Cremona and Sadek studied the case of elliptic curves over the rational numbers in which infinitely many \emph{admissible} local conditions may be imposed (including at the primes $2$ and $3$). In \cite{Phi22a}, the author extended the results of Cho and Jeong to the case of elliptic curves over arbitrary number fields satisfying finitely many local conditions (excluding conditions at primes above $2$ and $3$).  Using a results of the author in \cite{Phi22d}, one can derive results for densities of elliptic curves over arbitrary number fields satisfying infinitely many \emph{admissible} local conditions (again excluding conditions at primes above $2$ and $3$).  
In this article we study the analogous problem for Drinfeld modules:

\begin{question}\label{question:Drinfeld-local}
How many Drinfeld modules satisfy a given set of local conditions?
\end{question}

\subsection{Results for counting Drinfeld modules}

In this paper we will study Question \ref{question:Drinfeld-local} for Drinfeld $\F_q[T]$-modules over $\F_q(T)$. 
The local conditions considered will be good reduction, bad reduction, stable reduction, unstable reduction, stable reduction of rank $s$, and stable reduction of rank $\geq s$ (see Definition \ref{def:reduction-types} for precise definitions). 

For $\a\subseteq \F_q[T]$ an ideal let $N(\a):=\#(\F_q[T]/\a)$ denote the norm of $\a$. For $s\in \C$ with $\Re(s)> 1$ let
\[
\zeta_{\F_q(T)}(s):=\sum_{\a\subseteq \F_q[T]} N(\a)^{-s}
\]
denote the Dedekind zeta function of $\F_q(T)$ at $s$.

We will count Drinfeld modules up to a given height (see Definition \ref{def:Drinfeld-height} for the height we will be using). With respect to this height we will prove the following:

\begin{theorem}\label{thm:Drinfeld-Count}
The number of isomorphism classes of rank $r$ Drinfeld $\F_q[T]$-modules over $\F_q(T)$ with height equal to the positive integer $b$ is
\[
\frac{q^{r}}{\zeta_{\F_q(T)}\left(\frac{q^{r+1}-q}{q-1}\right)} q^{b\frac{q^{r+1}-q}{q-1}} +\begin{cases}
 O(q^b b) & \text{ if } r=2,\\
 O\left(q^{b\frac{q^{r+1}-q^2+q-1}{q-1}}\right) & \text{ if } r>2,
\end{cases}
\]
as a function of $b$.
\end{theorem}
 
The next result, which generalizes Theorem \ref{thm:Drinfeld-Count}, is our main result for counting Drinfeld modules when finitely many local conditions are imposed:

\begin{theorem}\label{thm:Drinfeld-finite-local-conditions}
Let $\p_1,\p_2,\dots,\p_m\subseteq \F_q[T]$ be a finite set of distinct prime ideals of $\F_q[T]$. For each $\p_i$ let $\msL_i$ be one of the local conditions in Table \ref{tab:LocalConditions}.  Then the number of isomorphism classes of rank $r$ Drinfeld $\F_q[T]$-modules over $\F_q(T)$ with reduction type $\msL_i$ at $\p_i$ for each $i$ and with height equal to the positive integer $b$ is
\[
\frac{q^{r}}{\zeta_{\F_q(T)}\left(\frac{q^{r+1}-q}{q-1}\right)} 
\left(\prod_{i=1}^m \kappa_{\msL_i}\right) q^{b\frac{q^{r+1}-q}{q-1}} +\begin{cases}
 O(q^b b) & \text{ if } r=2,\\
 O\left(q^{b\frac{q^{r+1}-q^2+q-1}{q-1}}\right) & \text{ if } r>2,
\end{cases}
\]
as a function of $b$, where $\kappa_{\msL}=\kappa'_{\msL}\frac{1}{1-N(\p)^{r-\frac{q^{r+1}-q}{q-1}}}$ with $\kappa'_{\msL}$ given in Table \ref{tab:LocalConditions}.
\end{theorem}

\begin{table}[h]
\centering
\begin{tabular}{cc}
\hline
$\msL$ & $\kappa'_\msL$  \\
\hline
stable of rank $s$ & $\frac{N(\p)^s-N(\p)^{s-1}}{N(\p)^r}$ \\
stable of rank $\geq s$ & $\frac{N(\p)^r-N(\p)^{s-1}}{N(\p)^r}$ \\
good & $\frac{N(\p)-1}{N(\p)}$  \\
bad & $\frac{1}{N(\p)}$ \\
stable & $\frac{N(\p)^r-1}{N(\p)^r}$   \\
unstable & $\frac{1}{N(\p)^r}$ \\
\hline
\end{tabular}
\caption{Local conditions}
\label{tab:LocalConditions}
\end{table}

The next result is our main result for counting Drinfeld modules when infinitely many local conditions are imposed:

\begin{theorem}\label{thm:Drinfeld-Infinite-Local-Conditions}
Let $r$ and $s$ be positive integers with $1\leq s<r$. Let $S=\{\p_1,\dots,\p_m\}$ be a finite set of distinct prime ideals of $\F_q[T]$. For each $\p_i$ let $\msL_i$ be one of the local conditions in Table \ref{tab:LocalConditions}. Then the number isomorphism classes of rank $r$ Drinfeld $\F_q[T]$-modules over $\F_q(T)$ which have stable reduction of rank $\geq s$ everywhere outside of $S$, and reduction type $\msL_i$ at $\p_i$ for each $\p_i\in S$, and of height equal to the positive integer $b$, is asymptotic to
\[
 \frac{\zeta_{\F_q(T)}\left(-r+\frac{q^{r+1}-q}{q-1}\right)}{\zeta_{\F_q(T)}(r-s+1)}\left(\prod_{\p_i\in S} \frac{\kappa'_{\msL_i}}{1-N(\p_i)^{-r+s-1}}\right) \frac{q^{r}}{\zeta_{\F_q(T)}\left(\frac{q^{r+1}-q}{q-1}\right)} q^{b\frac{q^{r+1}-q}{q-1}},
\]
as a function of $b$, where the $\kappa'_{\msL}$ are given in Table \ref{tab:LocalConditions}.
\end{theorem}

\begin{remark}
The factor $\frac{1}{1-N(\p)^{r-\frac{q^{r+1}-q}{q-1}}}$ appearing in the constant $\kappa_{\msL}$ accounts for the difference between counting equations for Drinfeld modules with a given local condition and counting isomorphism classes of Drinfeld modules. This quantity approaches 1 as $N(\p)$, $r$, or $q$ approaches infinity.
\end{remark}

\subsection{Examples}

The following examples illustrate how Theorem \ref{thm:Drinfeld-finite-local-conditions} and Theorem \ref{thm:Drinfeld-Infinite-Local-Conditions} can be used to determine the density of Drinfeld modules satisfying a prescribed set of local conditions.

\begin{example}
Consider the prime ideal $\p=(T^2+T+2)$ of $\F_3[T]$. Then the proportion of rank $r$ Drinfeld $\F_3[T]$-modules over $\F_3(T)$ with good reduction at $\p$ is 
\[
\frac{N(\p)-1}{N(\p)}\left(\frac{1}{1-N(\p)^{r-\frac{q^{r+1}-q}{q-1}}}\right) = \frac{8}{9}\left(\frac{1}{1-9^{r-\frac{3^{r+1}-3}{3-1}}}\right) \approx 88.89\%.
\]
Note that as $r\geq 2$, the term
\[
\frac{1}{1-N(\p)^{r-\frac{3^{r+1}-q}{3-1}}}\leq \frac{1}{1-9^{-10}}=90/89\approx 1.01
\]
has little effect on the density.
\end{example}

\begin{example}
Consider the prime ideal $\p=(T^6+T^4+4T^3+T^2+2)$ of $\F_5[T]$. Then the proportion of rank $4$ Drinfeld $\F_5[T]$-modules over $\F_5(T)$ with unstable reduction at $\p$ is 
\[
\frac{1}{N(\p)^4}\left(\frac{1}{1-N(\p)^{4-\frac{5^{5}-5}{5-1}}}\right) = \frac{1}{5^{24}}\left(\frac{1}{1-5^{-4656}}\right) \approx  0.00\%.
\]
\end{example}

\begin{example}
Consider the prime ideal $\p=(T+2)$ of $\F_7[T]$. Then the proportion of rank $5$ Drinfeld $\F_7[T]$-modules over $\F_7(T)$ with stable reduction of rank $\geq 3$ at $\p$ is 
\[
\frac{N(\p)^5-N(\p)^3}{N(\p)^5} \left(\frac{1}{1-N(\p)^{5-\frac{7^{6}-7}{7-1}}}\right)= \frac{7^5-7^3}{7^{5}}\left(\frac{1}{1-7^{-19602}}\right) \approx  97.96\%.
\]
\end{example}

\begin{example}
Consider the prime ideals $\p_1=(T)$, $\p_2=(T^2+T+2)$, and $\p_3=(T^2+2T+2)$ in $\F_3[T]$. Then the proportion of rank $3$ Drinfeld $\F_3[T]$-modules over $\F_3(T)$ with bad reduction at $\p_1$, stable reduction at $\p_2$, and stable reduction of rank $2$ at $\p_3$ is $\kappa_1\kappa_2\kappa_3$, where
\begin{align*}
\kappa_1&=\frac{1}{N(\p_1)}\left(\frac{1}{1-N(\p_1)^{3-\frac{3^{4}-3}{3-1}}}\right)=\frac{1}{3}\left(\frac{1}{1-3^{-36}}\right)\approx 33.00\%\\
\kappa_2&=\frac{N(\p_2)^3-1}{N(\p_2)^3}\left(\frac{1}{1-N(\p_2)^{3-\frac{3^{4}-3}{3-1}}}\right)=\frac{728}{729}\left(\frac{1}{1-9^{-36}}\right)\approx 99.86\%\\
\kappa_3&=\frac{N(\p_3)^2-N(\p_3)}{N(\p_3)^3}\left(\frac{1}{1-N(\p_2)^{3-\frac{3^{4}-3}{3-1}}}\right)=\frac{8}{81}\left(\frac{1}{1-9^{-36}}\right)\approx 9.89\%.\\
\end{align*}
In particular, we find that the density is
\[
\kappa_1\kappa_2\kappa_3\approx\frac{1}{3}\cdot \frac{728}{729}\cdot \frac{8}{81}
= \frac{5824}{177147}
\approx 3.29\%.
\]
\end{example}

\begin{example}
The proportion of rank $2$ Drinfeld $\F_{8}[T]$-modules over $\F_8(T)$ with everywhere stable reduction is 
\[
\frac{\zeta_{\F_8(T)}\left(-2+\frac{8^{2+1}-8}{8-1}\right)}{\zeta_{\F_8(T)}(2)}=\frac{\zeta_{\F_8(T)}(70)}{\zeta_{\F_8(T)}(2)}=\frac{\left(1-8^{-1}\right)\left(1-8^{-2}\right)}{\left(1-8^{-69}\right)\left(1-8^{-70}\right)}\approx\frac{441}{512} \approx  86.13\%.
\]
\end{example}

\begin{example}
The proportion of rank $11$ Drinfeld $\F_5[T]$-modules over $\F_5(T)$ with everywhere stable reduction of rank $\geq 9$ is 
\[
\frac{\zeta_{\F_5(T)}\left(-11+\frac{5^{11+1}-5}{5-1}\right)}{\zeta_{\F_5(T)}(11-9+1)}
=\frac{\zeta_{\F_5(T)}\left(61035144\right)}{\zeta_{\F_5(T)}(3)}=\frac{\left(1-5^{-2}\right)\left(1-5^{-3}\right)}{\left(1-5^{-61035143}\right)\left(1-5^{-61035144}\right)}\approx\frac{2976}{3125} \approx  95.23\%.
\]
\end{example}

\begin{example}
Consider the prime ideals $\p_1=(T+1)$, $\p_2=(T^2+T+1)$, and $\p_3=(T^3+T+1)$ in $\F_2[T]$. Then the proportion of rank $2$ Drinfeld $\F_2[T]$-modules over $\F_2(T)$ with bad reduction at $\p_1$, good reduction at $\p_2$ and $\p_3$, and stable reduction everywhere else, is 
\begin{align*}
\frac{\zeta_{\F_2(T)}\left(-2+\frac{2^{2+1}-2}{2-1}\right)}{\zeta_{\F_2(T)}(2)}\left(\frac{\kappa_1}{1-N(\p_1)^{-2}}\right)\left(\frac{\kappa_2}{1-N(\p_2)^{-2}}\right)\left(\frac{\kappa_3}{1-N(\p_3)^{-2}}\right),
\end{align*}
where
\begin{align*}
\kappa_1&= \frac{1}{N(\p_1)}=\frac{1}{2}\\
\kappa_2&=\frac{N(\p_2)-1}{N(\p_2)}=\frac{3}{4}\\
\kappa_3&=\frac{N(\p_3)-1}{N(\p_3)}=\frac{7}{8}.
\end{align*}
This, together with
\[
\frac{\zeta_{\F_2(T)}\left(-2+\frac{2^{2+1}-2}{2-1}\right)}{\zeta_{\F_2(T)}(2)}=\frac{\zeta_{\F_2(T)}\left(4\right)}{\zeta_{\F_2(T)}(2)}=\frac{(1-2^{-1})(1-2^{-2})}{(1-2^{-3})(1-2^{-4})}=\frac{16}{35},
\]
allows us to compute the density,
\begin{align*}
\frac{\zeta_{\F_2(T)}\left(-2+\frac{2^{2+1}-2}{2-1}\right)}{\zeta_{\F_2(T)}(2)}
\left(\frac{1/2}{1-2^{-2}}\right)\left(\frac{3/4}{1-4^{-2}}\right)\left(\frac{7/8}{1-8^{-2}}\right)
 &= \frac{16}{35}\cdot \frac{2}{3}\cdot \frac{4}{5}\cdot \frac{8}{9}
 = \frac{1024}{4725}
  \approx 21.67\%.
\end{align*}
\end{example}

\subsection{Counting points on weighted projective stacks} 

The methods used to prove our results for counting Drinfeld modules are similar to the methods used in \cite{Phi22a} and \cite{Phi22d}. Namely, we count points of bounded height on weighted projective stacks which satisfy prescribed local conditions (see Theorem \ref{thm:WProjFin} and Theorem \ref{thm:WProjInfty}), and then we exploit the fact that the (compactified) moduli stack of rank $r$ Drinfeld modules is a weighted projective stack. In the case of infinitely many local conditions, the geometric sieve is a key ingredient. 

Our results for counting rational points on weighted projective stacks satisfying finitely many or infinitely many local conditions,  Theorem \ref{thm:WProjFin} and Theorem \ref{thm:WProjInfty} respectively, may be of independent interest. They are function field analogs of some of the authors previous results for counting points of bounded height on weighted projective stacks over number fields \cite{Phi22a, Phi22d}. These results for counting points on weighted projective stacks over global function fields which satisfy local conditions, generalize previous results for counting points on projective spaces without local conditions (see work of DiPippo \cite{Dip90}, Wan \cite{Wan91}, and the author \cite{Phi22b}); and also for counting points on weighted projective stacks without local conditions (see \cite{Phi22c}).


\subsection{Data availability statement}

Data sharing is not applicable to this article as no datasets were generated or analyzed during the current study.

\subsection{Acknowledgments}
The author would like to thank Jordan Ellenberg for suggesting the Drinfeld module case, Bryden Cais for his guidance and encouragement, and the anonymous referees for pointing out many stylistic improvements. 

\section{Preliminaries}

In this section we will introduce notation related to global function fields and weighted projective stacks which will be used throughout the article.

\subsection{Global function fields}

Let $\mcC$ be a smooth, projective, geometrically connected, genus $g$ curve over the finite field $\F_q$. Let $K=\F_q(\mcC)$ denote the function field of $\mcC$. Let $\infty$ be a fixed closed point of $\mcC$ of degree $d_\infty$ over $\F_q$. Let $\mcA\subset K$ be the ring of functions regular outside of the point $\infty$. 
Let $\Val(K)$ denote the set of discrete valuations of $K$; note these are in bijection with the set of closed points on $\mcC$ (i.e., \textbf{places}). Let $v_\infty$ be the valuation corresponding to the closed point $\infty$, and let $\Val_\infty(K):=\{v_\infty\}$ denote the set of infinite places, and $\Val_0(K)=\Val(K)-\{v_\infty\}$ the set of finite places.

To each place $v\in \Val(K)$ let $K_v$ denote the completion of $K$ with respect to $v$, and define an absolute value on $K_v$ by
\[
|x|_v:=q^{-v(x)\deg(v)}.
\]
Let $\pi_v$ be a uniformizer at the place $v\in \Val(K)$, so that $|\pi_v|_v=q^{-\deg(v)}$.

\subsection{Weighted projective stacks}

Let $\A^n$ denote affine $n$-space, and let $\G_m$ denote the multiplicative group scheme.

\begin{definition}[Weighted projective stack]
Given an $(n+1)$-tuple of positive integers $\w=(w_0,w_1,\dots,w_n)$, the \textbf{weighted projective stack} $\mcP(\w)$ is defined to be the quotient stack
\[
\mcP(\w):=[(\A^{n+1}-\{0\})/\G_m],
\]
with respect to the action
\begin{align*}
\ast_\w\colon\mathbb{G}_m\times(\A^{n+1}-\{0\}) &\to (\A^{n+1}-\{0\})\\
(\lambda,(x_0,\dots,x_n)) &\mapsto \lambda\ast_\w (x_0,\dots,x_n):=(\lambda^{w_0}x_0,\dots, \lambda^{w_n}x_n).
\end{align*}
\end{definition}

As an example, the weighted projective stack with weights $\w=(1,1,\dots,1)$ coincides with projective $n$-space, $\P^n$. Later in this article we will use the fact that the compactified moduli stack of rank $r$ Drinfeld $\F_q[T]$-modules over $\F_q(T)$ is isomorphic to the weighted projective stack $\mcP(q-1,q^2-1,\dots,q^r-1)$.

\begin{remark}
The algebraic stack $\mcP(\w)$ is
 smooth (since it is a quotient stack of a smooth scheme by a smooth group scheme) and
proper (by the valuative criterion for stacks). The point $[(a_0,\dots,a_n)]\in\mcP(w_0,\dots,w_n)$ has stabilizer $\mu_m$ where $m=\gcd(w_i : a_i\neq 0)$. When $K$ is a field of characteristic relatively prime to $m$, we have that $\mu_m$ is finite and reduced over $K$. When $K$ is not relatively to $m$ then $\mu_m$ is not reduced over $K$. It follows that $\mcP(w_0,\dots,w_n)$ is a Deligne-Mumford stack over a field $K$ if and only if the characteristic of $K$ is relatively prime to each of the weights $w_i$. 
\end{remark}

\subsection{Heights on weighted projective stacks}

For each place $v\in \Val(K)$, let $P_v$ denote the corresponding closed point of $\mcC$.
Let $\DIV(\mcC)$ denote the set of divisors on the curve $\mcC$.
For $x$ an element of the function field $K$, let 
\[
\Div(x):=\sum_{v\in \Val(K)} v(x) P_v \in \DIV(\mcC) 
\]
denote the divisor of $x$. For $w\in \Z_{>0}$, define the \textbf{$w$-weighted divisor} of $x$ to be 
\[
\Div_w(x):=\sum_{v\in \Val(K)} \left\lfloor\frac{v(x)}{w}\right\rfloor P_v \in \DIV(\mcC).
\]

Let $\w=(w_0,\dots,w_n)$ be an $(n+1)$-tuple of positive integers.
For $x=[x_0:\cdots:x_n]\in \mcP(\w)(K)$, let $\inf_\w(x)$ denote the greatest divisor $D\in \DIV(\mcC)$ for which $D\leq \Div_{w_i}(x_i)$ for all $i$. 

\begin{definition}[Height]\label{def:WPS-height}
The \textbf{(logarithmic) height} of a point $x=[x_0:\cdots:x_n]\in \mcP(\w)(K)$ is defined as
\[
\Ht_{\w}(x):=-\deg(\inf\nolimits_\w(x))=-\sum_{v\in \Val(K)} \deg(v) \min_i\left\{\left\lfloor\frac{v(x_i)}{w_i}\right\rfloor\right\}.
\]
\end{definition}

\begin{remark}
This height coincides with the \emph{stacky height} associated to the tautological bundle of $\mcP(\w)$ \cite{ESZB23}.
\end{remark}

We may decompose the height into a finite part,
\[
\Ht_{\w,0}(x):=-\sum_{v\in \Val_0(K)} \deg(v) \min_i\left\{\left\lfloor\frac{v(x_i)}{w_i}\right\rfloor\right\},
\]
and an infinite part,
\[
\Ht_{\w,\infty}(x):=-d_\infty \min_i\left\{\left\lfloor\frac{v_\infty(x_i)}{w_i}\right\rfloor\right\},
\]
so that $\Ht_\w(x)=\Ht_{\w,0}(x)+\Ht_{\w,\infty}(x)$.

For $x=[x_0:\cdots:x_n]\in \mcP(\w)(K)$ define the \textbf{finite scaling divisor of $x$} by
\[
\mfD_0(x):=\sum_{v\in \Val_0(K)} \min_i\left\{\left\lfloor\frac{v(x_i)}{w_i}\right\rfloor\right\} P_v \in \DIV(\mcC).
\]
Then $-\deg(\mfD_0(x))=\Ht_{\w,0}(x)$.

We now define weighted versions of linear equivalence and Riemann--Roch spaces.

\begin{definition}[$w$-linearly equivalent divisors]
 Let $w\in\Z_{>0}$. Two divisors $D,D'\in \DIV(\mcC)$ are said to be \textbf{$w$-linearly equivalent} if $D-D'=\Div_w(x)$ for some $x\in \F_q(\mcC)$.
\end{definition} 


\begin{definition}[$w$-Riemann--Roch space]
For $D\in \DIV(\mcC)$ we define the \textbf{$w$-Riemann--Roch space} attached to $D$ as
\[
\mcL_w(D):=\{x\in \F_q(\mcC): \Div_w(x)+D\geq 0\} \cup \{0\}.
\]
Set $\ell_w(D):=\dim_{\F_q}(\mcL_w(D))$, the $\F_q$-dimension of the $w$-Riemann--Roch space.
\end{definition} 

When $w=1$ we recover the usual definition of the Riemann--Roch space, and in this case we will simply write $\mcL(D)$ for $\mcL_1(D)$ and $\ell(D)$ for $\ell_1(D)$.

\section{Weighted geometry of numbers over global function fields}

Let $v\in \Val(K)$ be a place of the global function field $K=\F_q(\mcC)$. Recall that $K_v$ is the completion of $K$ at the place $v$. Let $\mcA_v:=\{x\in K_v: v(x)\geq 0\}$ denote the ring of integers of $K_v$. Let $m_v$ denote the Haar measure on $K_v$, normalized so that $m_v(\mcA_v)=1$. This measure naturally extends to a measure on the vector space $K_v^n$, which we also denote by $m_v$. 

 The following result is a weighted version of the Principal of Lipschitz over global function fields, which follows from work of Bhargava, Shankar, and Wang \cite[Proposition 33]{BSW15}.

\begin{proposition}\label{prop:FunctionFieldLipschitz}
Let $R$ be an open compact subset of the vector space $K_\infty^n$. Let $\Lambda$ be a rank $n$ lattice in $K_\infty^n$, and let $m_\Lambda$ be the constant multiple of the measure $m_\infty$ for which $m_\Lambda(K_\infty^n/\Lambda)=1$. Let $\w=(w_1,\dots,w_n)$ be an $n$-tuple of positive integers, and set 
\[
|\w|:=w_1+\cdots+w_n
\]
equal to the sum of weights,
and
\[
w_{\min}:=\min\{w_1,\dots,w_n\}
\]
equal to the minimum weight.
For $t\in K_\infty$ set
\[
t\ast_\w R = \{t\ast_\w x : x\in R\}.
\]
Then
\[
\#\{\left(t\ast_\w R\right) \cap \Lambda\}
= m_\Lambda (R) |t|_\infty^{|\w|}  + O\left(|t|_\infty^{|\w|-w_{\min}}\right)=\frac{m_{\infty}(R)}{m_\infty(K_\infty^n/\Lambda)}|t|_\infty^{|\w|}  + O\left(|t|_\infty^{|\w|-w_{\min}}\right),
\]
as a function of $t$, where the implied constant depends only on $K$, $\infty$, $n$, $R$, $\w$, and $\Lambda$.
\end{proposition}

\begin{proof}
Since $m_\infty(t\ast_\w R)=|t|_\infty^{|\w|} m_\infty(R)$, \cite[Proposition 33]{BSW15} gives the desired leading term. Let $m_\infty(\proj(t\ast_\w R))$ denote the greatest $d$-dimensional volume obtained from the orthogonal projection of the subset $t\ast_\w R$ onto any $d$-dimensional coordinate subspace, for any $d<n$. Then the error term given in \cite[Proposition 33]{BSW15} is $O(m_\infty(\proj(t\ast_\w R)))$. In our case this equals $O\left(|t|_\infty^{|\w|-w_{\min}}\right)$, since the largest orthogonal projection is obtained by projecting onto any of the $(n-1)$-dimensional subspaces obtained by setting a coordinate with minimal weight, $w_{\min}:=\min_i\{w_i\}$, equal to zero.
\end{proof}

\begin{definition}[Boxes]
An \textbf{$(\mcA_v^n)$-box} is a subset $\mcB_v\subset \mcA_v^n$ for which there exist closed balls 
\[
\mcB_{v,j}=\{x\in \O_v : |x-a|_v \leq b_{v,j}\}\subset \mcA_v,
\]
 where $a\in \mcA_v$ and $b_{v,j}\in \{q_v^k : k\in \Z\}$, such that $\mcB_v$ is equal to the Cartesian product $\prod_{j=1}^n \mcB_{v,j}$. For $S\subset \Val_0(K)$ a finite set of finite places, an \textbf{$S$-box} is a subset $\mcB\subset \prod_{v\in S} \mcA_v^n$ for which there exists $\mcA_v^n$-boxes $\mcB_{v,j}$ such that
\[
\mcB=\prod_{v\in S} \mcB_v=\prod_{v\in S} \prod_{j=1}^n \mcB_{v,j}.
\]
\end{definition}

\begin{lemma}[Box Lemma]\label{lem:box}
Let $\Omega_{\infty}$ be an open compact subset of $K_{\infty}^n$, let $S\subset \Val_0(K)$ be a finite set of finite places, and let $\mcB=
\prod_{v\in S} \mcB_v$ be an $S$-box. Then, for each $t\in K_{\infty}$, we have
\begin{align*}
\#\left\{x\in \mcA^n\cap \left(t\ast_{\w} \Omega_{\infty}\right) : x\in \mcB \right\}
= \left( m_{\infty}(\Omega_\infty)\prod_{v\in S} m_{v}(\mcB_v)\right) |t|_\infty^{|\w|}+O\left(|t|_\infty^{|\w|-w_{\min}}\right),
\end{align*}
where the implied constant depends only on $K_{\infty}$, $n$, $\Omega_\infty$, and $\w$.
\end{lemma}

\begin{proof}
Notice that the intersection $\mcA^n \cap \mcB$, of the lattice $\mcA^n$ and the box $\mcB$, is a translate of a sub-lattice of $\mcA^n$ with determinant $m_\infty(K_\infty^n/(\mcA^n\cap\mcB))=\prod_{v\in S} m_\infty(\mcB_v)^{-1}$. Applying Proposition \ref{prop:FunctionFieldLipschitz} to this lattice proves the lemma.
\end{proof}

For $T$ a topological space and $S\subset T$ a subspace, let $\partial S$ denote the boundary of $S$.

\begin{lemma}\label{lem:O_F-Finite-Geometric-Sieve}
 Let $\Lambda$ be a rank $n$ $\mcA$-lattice in the vector space $K_\infty^n$. Set $\Lambda_\infty:=\Lambda\otimes_{\mcA}K_{\infty}$ and for any place $v\in \Val(K)$ set $\Lambda_v:=\Lambda\otimes_{\mcA} \mcA_v$. Let $\Omega_\infty\subset K_{\infty}^n$ be an open compact subset. For each place $v$ in a finite subset $S$ of $\Val_0(K)$, let $\Omega_v\subset \Lambda_v$ be a subset whose boundary has measure $m_v(\partial \Omega_v)=0$. Let $\w=(w_1,\dots,w_n)$ be an $n$-tuple of positive integers. Then
\[
\#\{x\in \Lambda\cap t\ast_\w \Omega_\infty: x\in \Omega_v \text{ for all } v\in S\}=\left(\frac{m_{\infty}(\Omega_\infty)}{m_{\infty}(\Lambda_{\infty}/\Lambda)} \prod_{v\in S} \frac{m_v(\Omega_v)}{m_v(\Lambda_v)}\right) |t|_\infty^{|\w|}+O\left(|t|_\infty^{|\w|-w_{\min}}\right).
\]
\end{lemma}

\begin{proof}
Fix an isomorphism $\Lambda\cong \mcA^n$. Let $m_{\Lambda_\infty}$ and $m_{\Lambda_v}$ denote the Haar measures on $\Lambda_\infty$ and $\Lambda_v$ respectfully, normalized so that $m_{\Lambda_\infty}(\Lambda_\infty)=1$ and $m_{\Lambda_v}(\Lambda_\infty)=1$. These measures induce measures on $K_{\infty}^n$ and $\mcA_v^{n}$, which differ from the usual Haar measures by $m_\infty(\Lambda_\infty/\Lambda)$ and $m_v(\Lambda_v)$ respectfully. It therefore suffices to prove the result in the case that $\Lambda=\mcA^n$. Note that if $\Lambda=\mcA^n$ and each $\Omega_v$ is an $\mcA_v$-box, then Lemma \ref{lem:O_F-Finite-Geometric-Sieve} is precisely the Box Lemma (Lemma \ref{lem:box}). However, if any of the $\Omega_v$ are not $\mcA_v$-boxes, then there is more work to be done.

Set $P:=\prod_{v\in S}\Omega_v$, and let
\[
Q:=\bigcup_{v\in S} \left((\mcA_v^n-\Omega_v)\prod_{u\in S-\{v\}} \mcA_u^n\right)
\]
 denote the complement of $P$ in $\prod_{v\in S}\mcA_v^n$. By assumption, the set $\partial \Omega_v=\partial (\mcA_v^n-\Omega_v)$ has measure zero. Therefore, by compactness, we may cover the  closure of $P$ by a finite set of boxes, $(\mcI^{(i)})_{i\in I}$, such that the sum of their measures is arbitrarily close to $\prod_{v\in S} m_\infty(\Omega_\infty)$. Similarly, we may cover the closure of $Q$ by a finite collection of boxes, $(\mcJ^{(j)})_{j\in J}$, such that the sum of their measures is arbitrarily close to $1-\prod_{v\in S}m_v(\Omega_v)$.

Applying the Box Lemma (Lemma \ref{lem:box}) to the boxes $(\mcI^{(i)})_{i\in I}$, and then summing over the boxes, gives an upper bound for the size of the set 
\begin{align}\label{eq:size_set}
\{x\in \Lambda\cap \left(t\ast_\w \Omega_\infty\right): x\in \Omega_v \text{ for all } v\in S\}.
\end{align}
Similarly, applying the Box Lemma to the boxes $(\mcJ^{(j)})_{j\in J}$, and then summing over the boxes, gives an upper bound for the size of the complement of the set (\ref{eq:size_set}). Together these bounds imply the desired asymptotic.
\end{proof}

For $d\in \Z_{\geq 0}$ a non-negative integer, let $\Val_d(K):=\{v\in \Val(K): \deg(v)\geq d\}$ denote the set of places of degree greater than or equal to $d$. Recall that $\pi_\infty$ is a uniformizer for the completion $K_\infty$ of the function field $K=\F_q(\mcC)$ at the closed point $\infty$. For real valued function $f(x)$ and $g(x)$, we write $f\sim g$ if $\lim_{n\to\infty} f(n)/g(n)=1$.

\begin{lemma}\label{lem:O_F-Infinite-Geometric-Sieve}
Let $\w=(w_1,\dots,w_n)\in \Z_{> 0}^n$ be an $n$-tuple of positive integers, let $\Lambda$ be an $\mcA$-lattice of rank $n$ in the vector space $K_\infty^n$, and set $\Lambda_\infty:=\Lambda\otimes_{\mcA} K_\infty$. Let $\Omega_\infty\subset \Lambda_\infty$ be an open compact subset whose boundary has measure $m_\infty(\partial \Omega_\infty)=0$. For each place $v\in \Val_0(K)$ set $\Lambda_v:=\Lambda\otimes_{\mcA}\mcA_v$, and let $\Omega_v\subset \Lambda_v$ be a subset whose boundary has measure $m_v(\partial\Omega_v)=0$. Assume that
\begin{equation}\label{eq:O_F-Infinite-Geometric-Sieve}
\lim_{d\to\infty} \limsup_{b \to \infty}\frac{\#\{ x\in \Lambda \cap \left(\pi_\infty^{-b} \ast_{\w} \Omega_\infty\right) : x\not\in \Omega_v \text{ for some } v\in \Val_d(K)\}}{q^{b|\w|}}=0.
\end{equation}
Then
\[
\#\{ x\in \Lambda \cap \left(\pi_\infty^{-b} \ast_\w \Omega_\infty\right) : x\in \Omega_v \text{ for all } v\in \Val_0(K)\} \sim \frac{m_\infty(\Omega_\infty)}{m_\infty(\Lambda_\infty/\Lambda)}\left(\prod_{v\in \Val_0(K)} \frac{m_v(\Omega_v)}{m_v(\Lambda_v)}\right) q^{b|\w|}.
\]

\end{lemma}

\begin{proof}
As in the proof of Lemma \ref{lem:O_F-Finite-Geometric-Sieve}, we may reduce to the case in which $\Lambda=\mcA^n$.

For $d\leq d'\leq \infty$ and $b>0$, define the function
\[
f_{d,d'}(b):=\frac{\#\{x\in \mcA^n \cap \left(\pi_\infty^{-b} \ast_\w \Omega_\infty\right) : x\in \Omega_v \text{ for all } v\in \Val_d(K)-\Val_{d'}(K)\}}{q^{b|\w|}},
\]
and set $f_d(b):=f_{1,d}(b)$. Note that for all $s\in \Z_{\geq 0}\cup \{\infty\}$ we have  $f_{d}(b)\geq f_{d+s}(b)$. The hypothesis (\ref{eq:O_F-Infinite-Geometric-Sieve}) implies that
\begin{align}\label{eq:1-pf-O_F-infinite-geometric-sieve}
\lim_{d\to\infty}\limsup_{b \to\infty} \left(f_d(b)-f_\infty(b)\right)=0.
\end{align}
For all $d<d'<\infty$ we have, by Lemma \ref{lem:O_F-Finite-Geometric-Sieve}, that
\begin{align}\label{eq:2-pf-O_F-infinite-geometric-sieve}
\lim_{b\to\infty} f_{d,d'}(b)=m_\infty(\Omega_\infty)\prod_{v\in \Val_{d}(K)-\Val_{d'}(K)} m_v(\Omega_v).
\end{align}
Combining the limits (\ref{eq:1-pf-O_F-infinite-geometric-sieve}) and (\ref{eq:2-pf-O_F-infinite-geometric-sieve}) gives
\[
\lim_{b\to\infty} f_{\infty}(b)=\lim_{d\to\infty}\lim_{b\to\infty} f_d(b)=m_\infty(\Omega_\infty)\prod_{v\in \Val_d(K)} m_v(\Omega_v).
\]
Note that the infinite product converges by Cauchy's criterion, since the limits (\ref{eq:O_F-Infinite-Geometric-Sieve}) and (\ref{eq:2-pf-O_F-infinite-geometric-sieve}) together imply that
\[
\lim_{d\to\infty} \sup_{s\geq 1} \left| 1-\prod_{v\in \Val_d(K)-\Val_{d+s}(K)} m_v(\Omega_v)\right| =\frac{1}{m_\infty(\Omega_\infty)}\lim_{d\to\infty} \sup_{s\geq 1} \lim_{b\to\infty} |f_1(b)-f_{d, d+s}(b)|=0.
\]
\end{proof}

As an important example, note that $\mcA^n$ may naturally be viewed as a lattice in $K_\infty^n$. Following Weil \cite[2.1.3 (b)]{Wei82}, we compute the covolume $m_\infty(K_\infty^n/\mcA^n)$. First, note that $m_\infty(K_\infty^n/\mcA^n)=m_\infty(K_\infty/\mcA)^n$. Since the uniformizer $\pi_\infty$ has a zero at $\infty$, it must have a pole somewhere in $\mcC-\infty$, and thus $\mcA\cap (\pi_\infty)=\{0\}$. Therefore the quotient $K_\infty/\mcA$ has a set of representatives consisting of cosets of $(\pi_\infty)$; more explicitly,
\[
K_\infty/\mcA=\bigsqcup_{F\in \{f\in K\ :\ \Div(f)\geq \infty\}} F \cdot (\pi_\infty).
\]
 By Riemann-Roch, $\#\{f\in K\ :\ \Div(f)\geq \infty\}=q^{d_\infty+g-1}$. As $m_\infty(\pi_\infty\mcA)= q^{-d_\infty}$, it follows that
 \[
 m_\infty(K_\infty^n/\mcA^n)=q^{n(g-1)}.
 \]
More generally, the above argument can be modified to show that, if $\a\subseteq \mcA$ is an integral ideal of $\mcA$, then the lattice
\[
\Lambda_\a:=\a^{w_1}\times\cdots\times \a^{w_n}\subset K_\infty^n
\]
has covolume
\[
m_\infty(K_\infty^n/\Lambda_\a)=N(\a)^{|\w|} q^{\#\w(g-1)},
\]
where $N(\a)=\#(\mcA/\a)$ is the norm of $\a$. By Proposition \ref{prop:FunctionFieldLipschitz} we obtain the following result:

\begin{proposition}\label{prop:FunctionFieldIdealLipschitz}
Let $R$ be an open compact subset of $K_\infty^n$, let $\a\subset \mcA$ be an integral ideal, and let $\Lambda_\a$ denote the rank $n$ lattice $\a^{w_1}\times\cdots\times \a^{w_n}$ in $K_\infty^n$.
Then
\[
\#\{(t\ast_\w R) \cap \Lambda_\a\}
=\frac{m_{\infty}(R)}{N(\a)^{|\w|} q^{\#\w(g-1)}}|t|_\infty^{|\w|}  + O\left(|t|_\infty^{|\w|-w_{\min}}\right) ,
\]
as a function of $t$, where the implied constant depends only on $K$, $\infty$, $n$, $R$, $\w$, and $\Lambda$.
\end{proposition}

By Lemma \ref{lem:O_F-Finite-Geometric-Sieve}, we have the following proposition:

\begin{proposition}\label{prop:O_F-Finite-Geometric-Sieve}
Let $\Omega_\infty\subset K_\infty^n$ be an open compact subset, and let $\a\subset \mcA$ be an integral ideal. For each place $v$ in a finite subset $S$ of $\Val_0(K)$, let $\Omega_v\subset \mcA_v^n$ be a subset whose boundary has measure $m_v(\partial \Omega_v)=0$.
Let $\w$ be an $n$-tuple of positive integers.
Then
\begin{align*}
&\#\{x\in \Lambda_\a \cap (t \ast_\w \Omega_\infty) : x\in \Omega_v \text{ for all places } v\in S\}\\
& \hspace{1cm} = \frac{m_\infty(\Omega_\infty)}{N(\a)^{|\w|} q^{\#\w(g-1)}} \left(\prod_{v \in S} m_v(\Omega_v)\right) |t|_\infty^{|\w|} + O(|t|_\infty^{|\w|-w_{\min}}),
\end{align*}
as a function of $t$, where the implied constant depends only on $K$, $\Omega_\infty$, $\Omega_v$, and $\w$.
\end{proposition}

By Lemma \ref{lem:O_F-Infinite-Geometric-Sieve} we have the following proposition:

\begin{proposition}\label{prop:K-Infinite-Geometric-Sieve}
Let $\Omega_\infty\subset K_\infty^n$ be an open compact subset, and let $\a\subset \mcA$ be an integral ideal. For each place $v$ in $\Val_0(K)$, let $\Omega_v\subset \mcA_v^n$ be a subset whose boundary has measure $m_v(\partial \Omega_v)=0$.
Let $\w$ be an $n$-tuple of positive integers.
Assume that 
\begin{equation}\label{eq:K-Infinite-Geometric-Sieve}
\lim_{d\to\infty} \limsup_{b \to \infty}\frac{\#\{ x\in \Lambda_\a \cap (\pi_\infty^{-b} \ast_{\w} \Omega_\infty) : x\not\in \Omega_v \text{ for some } v\in \Val_d(K)\}}{q^{b|\w|}}=0.
\end{equation}
Then
\begin{align*}
&\#\{x\in \Lambda_\a \cap (\pi_\infty^{-b} \ast_\w \Omega_\infty) : x\in \Omega_\p \text{ for all } v\in \Val_0(K)\}\\
&\hspace{1cm}\sim \frac{m_\infty(\Omega_\infty)}{N(\a)^{|\w|} q^{\#\w(g-1)}} \left(\prod_{v \in \Val_0(K)} m_v(\Omega_v)\right) q^{b|\w|}.
\end{align*}
\end{proposition}

The following lemma gives a useful criterion for checking the limit condition (\ref{eq:K-Infinite-Geometric-Sieve}) of Proposition \ref{prop:K-Infinite-Geometric-Sieve}.

\begin{lemma}\label{lem:Ek-Bha}
Let $K$ be a global function field and $Y\subset \A_{\mcA}^n$ a closed subscheme of codimension $c>1$. Let $\Omega_\infty\subset K_\infty^n$ be a bounded subset whose boundary has measure $m_\infty(\partial \Omega_\infty)=0$ and $m_\infty(\Omega_\infty)>0$. 
For each $v\in \Val_0(K)$, set
\[
\Omega_v = \{x\in \mcA_v^n : x \pmod{ \p_v} \not\in Y(\F_{q_v})\}.
\]
Then
\[
\# \{x\in \mcA^n \cap \pi_\infty^{-b} \ast_\w \Omega_\infty : x \pmod{\p_v} \not\in \Omega_v \text{ for some } v\in \Val_M(K)\} = O\left(\frac{q^{b|\w|}}{M^{c-1}\log(M)}\right),
\]
as a function of $b$, where the implied constant depends only on $\Omega_\infty$ and $Y$. In particular, the limit condition (\ref{eq:K-Infinite-Geometric-Sieve}) holds in this situation.
\end{lemma}

\begin{proof}
The case $K=\Q$ is due to Bhargava \cite[Theorem 3.3]{Bha14} (which generalized a result of Ekedahl \cite{Eke91}). Bhargava's method generalizes to arbitrary global fields (as noted in \cite[Theorem 21]{BSW15}) and to arbitrary weights (as noted in \cite[pg.~4]{BSW22}).
\end{proof}

\section{Counting points on weighted projective stacks}\label{sec:PROJ}

In this section we use the results of the previous section, concerning weighted geometry of numbers over global function fields, to prove asymptotics for the number of points of bounded height on weighted projective stacks satisfying prescribed local conditions.

 We maintain the notation of the previous sections. In particular, let $\mcC$ be a smooth, projective, geometrically connected genus $g$ curve over the finite field $\F_q$ and let $K$ be the function field of $\mcC$. Denote by $h_K$ the class number of $K$.
  Let $\w=(w_0,\dots, w_n)$ be an $(n+1)$-tuple of positive integers. Set $\# \w:=n+1$, $|\w|:=w_0+\cdots+w_n$, and $w_{\min}:=\min\{w_0,\dots,w_n\}$.
  
 For any field $F$ and any subset $\Omega\subseteq \mcP(\w)(F)$, let $\Omega^{\aff}\subset F^{n+1}$ denote the subset of the affine cone of $\mcP(\w)(F)$ above $\Omega$ (i.e., the preimage of $\Omega$ with respect to the map $(F^{n+1}-\{0\})\to \mcP(\w)(F)$).

\subsection{Finitely many local conditions}

We first address the case in which finitely many local conditions are imposed.

\begin{theorem}\label{thm:WProjFin}
Let $\Omega_\infty\subseteq \mcP(\w)(K_\infty)$, and suppose that $\Omega_\infty^{\aff}\subset K_\infty^{n+1}$ is an open compact subset.
 Let $S\subset \Val_0(K)$ be a finite set of finite places, and for each $v\in S$ let $\Omega_v\subset \mcP(\w)(K_v)$ be a subset for which $m_v(\partial \Omega_v^{\aff})=0$. 
For $x\in \mcP(\w)(K)$, write $x\in \Omega$ if $x\in \Omega_v$ for all $v\in \Val(K)$.
Then, for $b\in \Z_{\geq 1}$, we have that
\begin{align*}
\#\{x\in \mcP(\w)(K): \Ht_\w(x) = b,\ x\in \Omega\}
= \kappa q^{|\w|b}
 + \begin{cases}
O\left(q^b b\right) &  \text{ if } n=d_\infty=1\\
O\left(q^{b(|\w|-w_{\min}/d_\infty)}\right) & \text{ else, }
\end{cases}
\end{align*}
as a function of $b$, where the leading coefficient is
\[
\kappa=\frac{h_{K}\gcd(q-1,\w)}{\zeta_K(|\w|)q^{\#\w(g-1)}(q-1)} \left(\frac{m_\infty(\{x\in \Omega_\infty^{\aff} : \Ht_\infty(x)= 0\})}{m_\infty(\{ x\in K_\infty^{n+1} : \Ht_\infty(x)=0\})}\right)
\left(\prod_{v\in S} m_v\left(\Omega_v^{\aff}\cap \mcA_v^{n+1}\right)\right).
\]
\end{theorem}

\begin{proof}

Let $h=h_\mcA$ denote the size of the class group $\Pic(\mcA)$ of the Dedekind domain $\mcA$, and let $D_1,\dots,D_h$ be a set of effective divisor representatives of $\Pic(\mcA)$. Then we have the following partition of $\mcP(\w)(K)$ into points whose finite scaling divisors are in the same divisor class:
\[
\mcP(\w)(K)=\bigsqcup_{i=1}^h \{x\in \mcP(\w)(K): [\mfD_0(x)]=[D_i]\}.
\] 
For each $D\in \{D_1,\dots,D_h\}$, consider the counting function
\[
M(D,b):=\#\{x\in \mcP(\w)(K) : \Ht_\w(x)= b,\ [\mfD_0(x)]=[D],\ x\in \Omega\}.
\]
Consider the action of the unit group $\mcA^\times=\F_{q}^\times$ on $K^{n+1}-\{0\}$ given by
\[
u\ast_\w (x_0,\dots,x_n):=(u^{w_0} x_0,\dots, u^{w_n}x_n),
\]
 and let $(K^{n+1}-\{0\})/\mcA^\times$ denote the corresponding set of orbits. 
We may describe $M(D,b)$, in terms of $\mcA^\times$-orbits in the affine cone of $\mcP(\w)(K)$, as follows: There is a bijection between the sets
\[
\{x=[x_0:\cdots:x_n]\in \mcP(\w)(K): \Ht_\w(x)=b,\ [\mfD_0(x)]=[D],\ x\in \Omega\}
\]
and
\[
 \{x=[(x_0,\dots,x_n)]\in (K^{n+1}-\{0\})/\mcA^\times : \Ht_\infty(x) - \deg(D)=b,\ \mfD_0(x)=D,
  x\in \Omega^{\aff}\}
 \]
 given by
 \[
 [x_0:\dots:x_n] \mapsto [(x_0,\dots,x_n)] .
 \]
From this it follows that
\[
M(D,b)=\#\{x\in (K^{n+1}-\{0\})/\mcA^\times :\Ht_{\infty}(x) - \deg(D)=b,\ \mfD_0(x)=D,\ x\in \Omega^\aff\}.
\]
In order to obtain an asymptotic for $M(D,b)$, we will first find an asymptotic for the related counting function,
\[
M'(D,b)=\#\{x\in (K^{n+1}-\{0\})/\mcA^\times : \Ht_{\infty}(x) - \deg(D) = b,\ D\leq\mfD_0(x),\ \Omega^{\aff}\},
\]
and then use M\"obius inversion.

Our first step towards finding an asymptotic for $M'(D,b)$ is to construct a fundamental domain for the ($\w$-weighted) action of the unit group $\mcA^{\times}$ on $K_\infty^{n+1}-\{0\}$. Define the open bounded sets
\[
\mcF(b):=\{x\in K_\infty^{n+1}-\{0\} : \Ht_\infty(x)=b\}.
\]
These sets are $\mcA^\times$-stable, in the sense that if $u\in \mcA^\times$ and $x\in \mcF(b)$ then $u\ast_\w x\in \mcF(b)$; this can be verified as follows:
\begin{align*}
\Ht_{\infty}(u\ast_{\w} x)
&= -d_\infty \min_i\left\{\left\lfloor\frac{v_\infty(u\ast_\w x)}{w_i}\right\rfloor\right\}\\
&=-d_\infty \min_i\left\{ \left\lfloor v_\infty \left(u \right)+\frac{v_\infty(x)}{w_i}\right\rfloor\right\}\\
&=-d_\infty \min_i\left\{\left\lfloor\frac{v_\infty(x)}{w_i}\right\rfloor\right\}\\
&=\Ht_{\infty}(x).
\end{align*}
Similarly, for any $t\in K_\infty$, we have that
\begin{align*}
\Ht_{\infty}(t\ast_{\w} x)
&= -d_\infty \min_i\left\{\left\lfloor\frac{v_\infty(t\ast_\w x_i)}{w_i}\right\rfloor\right\}\\
&= -d_\infty \min_i\left\{\left\lfloor v_\infty\left(t\right)+\frac{v_\infty(x_i)}{w_i}\right\rfloor\right\}\\
&= -d_\infty \left( v_\infty(t)+\min_i\left\{\left\lfloor\frac{v_\infty(x_i)}{w_i}\right\rfloor\right\}\right)\\
&=\Ht_{\infty}(x) - d_\infty v_\infty(t).
\end{align*}

 Note that the image of the map $\Ht_{\infty}\colon K_{\infty}\to \R$ is $\{{d_\infty n} : n\in \Z\}$. Thus the set $\mcF(b)$ is empty if $b\not\in d_\infty \Z$. When $b\in d_\infty \Z$, we have that
\[
\mcF(b)=\pi_\infty^{-b/d_\infty }\ast_{\w}\mcF(0).
\]
 
Let $\d\subset \mcA$ denote the ideal of the ring $\mcA$ corresponding to the divisor $D$, and define the lattice 
\[
\Lambda_D:=\d^{w_0}\times\cdots\times \d^{w_n}\subset K_\infty^{n+1}.
\]
Observe that the condition that the element $x\in K^{n+1}-\{0\}$ satisfies $D\leq \mfD_0(x)$, is equivalent to the condition that $x\in \Lambda_D-\{0\}$. Therefore
\[
M'(D,b)=\#\{x\in (\Lambda_D-\{0\})/\mcA^\times : \Ht_\infty(x)-\deg(D)=b,\ x\in \Omega^\aff\}.
\]
As the roots of unity of $K$ are $\F_{q}^\times$, and each $\F_{q}^\times$-orbit (with respect to the $\w$-weighted action) of an element of $(K-\{0\})^{n+1}$ contains exactly $(q-1)/\gcd(q-1, w_0,\dots,w_n)$ elements, we have that
\begin{align*}
M'(D,b)=\frac{\gcd(q-1,\w)}{q-1} \#\{\mcF(b+\deg(D)) \cap \Lambda_D \cap \Omega^{\aff}\}
+O\left(q^{b(|\w|-w_{\min}/d_\infty)}\right),
\end{align*}
where the error term accounts for the points counted by $M'(D,b)$ which are contained in the subvariety of the affine cone of $\mcP(\w)(K)$ consisting of points with at least one coordinate equal to zero. This subset can be estimated using Proposition \ref{prop:O_F-Finite-Geometric-Sieve}.


If $b\not\equiv -\deg(D)\pmod{d_\infty}$, then $\mcF(b)=0$, and thus $M'(D,b)=0$. On the otherhand, when $b\equiv -\deg(D) \pmod{d_\infty}$, we may applying Proposition \ref{prop:O_F-Finite-Geometric-Sieve} with the open compact set $\Omega_\infty^\aff\cap\mcF(0)$ and $\a=\d$, to obtain the following asymptotic:
\begin{equation}\label{eq:I<=c asymptotic}
\begin{aligned}
&M'(D,b)\\
&\vspace{5mm} =\frac{m_\infty(\Omega_\infty^{\aff}\cap \mcF(0))\gcd(q-1,\w)}{N(\d)^{|\w|}q^{\#\w (g-1)}(q-1)} \left(\prod_{v\in S} m_v\left(\Omega_v^{\aff}\cap \mcA_v^{n+1}\right)\right) (q^b N(\d))^{|\w|}+ O\left(q^{b(|\w|-w_{\min}/d_\infty)}\right)\\
&\vspace{5mm} =\frac{m_\infty(\Omega_\infty^\aff\cap\mcF(0))\gcd(q-1,\w)}{q^{\#\w(g-1)}(q-1)} \left(\prod_{v\in S} m_v\left(\Omega_v^{\aff}\cap \mcA_v^{n+1}\right)\right) q^{b|\w|}+ O\left(q^{b(|\w|-w_{\min}/d_\infty)}\right).
\end{aligned}
\end{equation}
Note that
\[
\frac{m_\infty(\Omega_\infty^{\aff}\cap \mcF(0))}{m_\infty(\mcF(0))}=\frac{m_\infty(\{x\in \Omega_\infty^{\aff} : \Ht_\infty(x)= 0\})}{m_\infty(\{ x\in K_\infty^{n+1} : \Ht_\infty(x)=0\})}.
\]
As $m_\infty(\mcF(0))= 1-q^{-d_\infty|\w|}$, we have 
\[
m_\infty(\Omega_\infty^{\aff}\cap \mcF(0))=(1-q^{-d_\infty|\w|}) \frac{m_\infty(\{x\in \Omega_\infty^{\aff} : \Ht_\infty(x)= 0\})}{m_\infty(\{ x\in K_\infty^{n+1} : \Ht_\infty(x)=0\})}.
\]
Set $\kappa'$ equal to the leading coefficient of the asymptotic (\ref{eq:I<=c asymptotic}), so that
\[
\kappa'=\frac{\left(1-q^{-d_\infty|\w|}\right) \gcd(q-1,\w)}{q^{\#\w(g-1)}(q-1)} 
 \left(\frac{m_\infty(\{x\in \Omega_\infty^{\aff} : \Ht_\infty(x)= 0\})}{m_\infty(\{ x\in K_\infty^{n+1} : \Ht_\infty(x)=0\})}\right)
\left(\prod_{v\in S} m_v\left(\Omega_v^{\aff}\cap \mcA_v^{n+1}\right)\right).
\]
 Let $\DIV^+(\mcC)$ denote the set of effective divisors on $\mcC$. For $x\in \Lambda_D$ we have $\mfD_0(x)=D+D'$ for some effective divisor $D'\in \DIV^+(\mcC)$. Then
\[
M'(D,b)=\sum_{D'\in \DIV^+(\mcC)} M(D+D', b-\deg(D'))= \sum_{\substack{D'\in \DIV^+(\mcC)\\ \deg(D') \leq b}} M(D+D', b-\deg(D')),
\]
where the final equality follows from the fact that $M(D,b)$ is zero for $b<0$.

We now apply M\"obius inversion and use our asymptotic (\ref{eq:I<=c asymptotic}) for $M'(D,b)$. For $b\equiv -\deg(D)\pmod{d_\infty}$ we have
\begin{align*}
M(D,b) &= \sum_{\substack{D'\in \DIV^+(\mcC) \\ \deg(D') \leq b}} \mu(D') M'(D+D', b-\deg(D'))\\
&=\sum_{\substack{D'\in \DIV^+(\mcC) \\ \deg(D') \leq b}} \mu(D')\left(\kappa'\left(q^{b-\deg(D')}\right)^{|\w|}+ O\left(\left(q^{b-\deg(D')}\right)^{|\w|-w_{\min}/d_\infty}\right)\right)\\
&=\kappa' q^{b|\w|}
\left(\sum_{D'\in \DIV^+(\mcC)} \mu(D')q^{-|\w|\deg(D')}
 - \sum_{\substack{D'\in \DIV^+(\mcC) \\ \deg(D') \leq b}}\mu(D')q^{-|\w|\deg(D')}\right)\\ 
&\hspace{20mm} +O\left(q^{b(|\w|-w_{\min}/d_\infty)} \sum_{\substack{D'\in \DIV^+(\mcC) \\ \deg(D') \leq b}}q^{-(|\w|-w_{\min}/d_\infty)\deg(D')}\right)\\
&=\kappa' q^{b|\w|}\left(\frac{1}{\zeta_\mcA(|\w|)}
 - O\left(q^{-bn}\right)\right) 
+ \begin{cases}
O\left(q^{b}b\right) & \text{ if } n=d_\infty=1\\
O\left(q^{b(|\w|-w_{\min}/d_\infty)}\right) & \text{ else, }
\end{cases}\\
&= \frac{\kappa'}{\zeta_\mcA(|\w|)} q^{b|\w|}
 + \begin{cases}
O\left(q^{b}b\right) & \text{ if } n=d_\infty=1\\
O\left(q^{b(|\w|-w_{\min}/d_\infty)}\right) & \text{ else. }
\end{cases}
\end{align*}
When $b\not\equiv -\deg(D) \pmod{d_\infty}$, then $M'(D,b)=0$, and thus $M(D,b)=0$.

Let $\Pic^0(\mcC)$ denote the finite group of degree zero divisors on $\mcC$ modulo divisors of elements of the units $K^\times$ of $K$. From the exact sequence 
\[
0 \rightarrow \Pic^0(\mcC) \rightarrow \Pic(\mcA) \xrightarrow{\deg} \Z/d_\infty\Z \rightarrow 0,
\]
we have that for any $a\in \Z/d_\infty \Z$ there are exactly $h_K=\#\Pic^0(\mcC)$ elements of $\Pic(\mcA)$ whose degree is congruent to $a$ modulo $d_\infty$. Therefore, summing over the divisor class representatives $D_i$ of $\Pic(\mcA)$ gives
\begin{align*}
\#\{x\in \P^n(K): \Ht(x) = b,\ x\in \Omega\}
&= \sum_{i=1}^{h} M(D_i,b)\\
&= \kappa q^{b|\w|}
 + \begin{cases}
O\left(q^{b}b\right) & \text{ if } n=d_\infty=1\\
O\left(q^{b(|\w|-w_{\min}/d_\infty)}\right) & \text{ else, }
\end{cases},
\end{align*}
where
\begin{align*}
\kappa&=\frac{\kappa' h_K}{\zeta_\mcA(|\w|)}\\
&=\frac{h_{K}\gcd(q-1,\w)}{\zeta_K(|\w|)q^{\#\w(g-1)}(q-1)} \left(\frac{m_\infty(\{x\in \Omega_\infty^{\aff} : \Ht_\infty(x)= 0\})}{m_\infty(\{ x\in K_\infty^{n+1} : \Ht_\infty(x)=0\})}\right)
\left(\prod_{v\in S} m_v\left(\Omega_v^{\aff}\cap \mcA_v^{n+1}\right)\right).
\end{align*}
\end{proof}

\subsection{Infinitely many local conditions}

We now count points on weighted projective stacks satisfying infinitely many local conditions.

\begin{theorem}\label{thm:WProjInfty}
Let $\Omega_\infty\subseteq \mcP(\w)(K_\infty)$ be a subset for which $\Omega_\infty^{\aff}\subset K_\infty^{n+1}$ is an open compact subset. For each $v\in \Val_0(K)$ let $\Omega_v\subset \mcP(\w)(K_v)$ be a subset for which $m_v(\partial \Omega_v^{\aff})=0$. Assume also that, for any open compact set $\Psi\subset K_\infty^{n+1}$, we have 
\begin{align}\label{eq:WProjInfty}
\lim_{d\to\infty}\limsup_{b\to\infty} \frac{\#\{x\in\mcA^{n+1}\cap \left(\pi_\infty^{-b} \ast_\w \Psi\right) : x\notin \Omega_v \text{ for some } v\in \Val_0(K) \text{ with } \deg(v)> d
\}}{q^{b|\w|}}=0.
\end{align}
Then
\begin{align*}
\#\{x\in \mcP(\w)(K): \Ht_\w(x) = b,\ x\in \Omega\}
\sim \kappa q^{b|\w|},
 \end{align*}
as a function of $b$, where the leading coefficient is
\[
\kappa=\frac{h_{K}\gcd(q-1,\w)}{\zeta_K(|\w|)q^{\#\w(g-1)}(q-1)} \left(\frac{m_\infty(\{x\in \Omega_\infty^{\aff} : \Ht_\infty(x)= 0\})}{m_\infty(\{ x\in K_\infty^{n+1} : \Ht_\infty(x)=0\})}\right)
\left(\prod_{v\in \Val_0(K)} m_v\left(\Omega_v^{\aff}\cap \mcA_v^{n+1}\right)\right).
\]
\end{theorem}

\begin{proof}
The proof is the same as the proof of Theorem \ref{thm:WProjFin}, but with the use of Proposition \ref{prop:O_F-Finite-Geometric-Sieve} replaced by Proposition \ref{prop:K-Infinite-Geometric-Sieve}. For any divisor 
\[
D=\sum_{v\in \Val(K)} n_v P_v\in \DIV(\mcC),
\]
set $D_v$ equal to $n_v$. In order to estimate
\begin{align*}
M'(D,b)=\frac{\gcd(q-1,\w)}{q-1} \#\{\mcF(b+\deg(D)) \cap \Lambda_D \cap \Omega^{\aff}\}
+O\left(q^{b(|\w|-w_{\min}/d_\infty)}\right)
\end{align*}
using Proposition \ref{prop:K-Infinite-Geometric-Sieve}, one must check that the local conditions
\begin{align*}
\Theta_v &:=\{x\in \Omega_v^{\aff} \cap \mcA_v^{n+1} : \mfD_0(x)_v=D_v\}
\end{align*}
satisfy hypothesis (\ref{eq:K-Infinite-Geometric-Sieve}) of Proposition \ref{prop:K-Infinite-Geometric-Sieve}.

Applying de Morgan's laws to the assumption (\ref{eq:WProjInfty}), we may reduce the problem of verifying the limit condition (\ref{eq:WProjInfty}) for the $\Theta_v$, to verifying the limit condition (\ref{eq:K-Infinite-Geometric-Sieve}) for the local conditions
\[
\Theta'_v :=\{x\in \mcA_v^{n+1} : \mfD_0(x)_v=D_v\}.
\]
Applying Lemma \ref{lem:Ek-Bha} with the subscheme $x_0=x_1=\cdots=x_n=0$ shows that the sets
\[
\Theta''_v:=\{x\in \mcA_v^{n+1} : \mfD_0(x)_v=0\}
\]
satisfy (\ref{eq:K-Infinite-Geometric-Sieve}). But $\Theta''_v = \Theta'_v$ for all $v$ of sufficiently large degree, and thus the $\Theta'_v$ must also satisfy (\ref{eq:K-Infinite-Geometric-Sieve}).
\end{proof}

The next lemma follows immediately from Lemma \ref{lem:Ek-Bha}, and gives a useful tool for checking the limit condition (\ref{eq:WProjInfty}) of Theorem \ref{thm:WProjInfty}.

\begin{lemma}\label{lem:WPROJ-Ek-Bha}
Let $K$ be a global function field and $Y\subset \mcP_{\mcA}(\w)$ a closed substack of codimension $c>1$. For each $v\in \Val_0(K)$, set
\[
\Omega_v = \{x\in \mcP(\w)(\mcA) : x \pmod{ \p_v} \not\in Y(\F_{q_v})\}.
\]
 Then, for all bounded subsets $\Psi\subset K_\infty^{n+1}$ with $m_\infty(\partial \Psi)=0$ and $m_\infty(\Psi)>0$, we have that
\[
\# \{x\in \mcA^n \cap (\pi_\infty^{-r} \ast_\w \Psi) : x \pmod{\p_v} \not\in \Omega_v^\aff \text{ for some } v\in \Val_M(K)\} = O\left(\frac{q^{r|\w|}}{M^{c-1}\log(M)}\right),
\]
as a function of $r$, where the implied constant depends only on $\Psi$, $\w$, and $Y$. In particular, (\ref{eq:WProjInfty}) holds in this situation.
\end{lemma}

\section{Counting Drinfeld Modules}

In this section we apply Theorem \ref{thm:WProjFin} and Theorem \ref{eq:WProjInfty} to moduli stacks of Drinfeld modules. This gives results for counting Drinfeld modules with prescribed local conditions. Throughout this section fix $\mathcal{A}=\F_q[T]$ and $K=\F_q(T)$.

\subsection{Preliminaries on Drinfeld modules}

Let $K\{\tau\}$ denote the non-commutative polynomial ring in $\tau$ with coefficients in $K$, which satisfies the commutation relations $\tau c=c^q\tau$ for all $c\in K$. 

\begin{definition}[Drinfeld module]
A \textbf{Drinfeld $\mcA$-module of rank $r$ defined over $K$} is a ring homomorphism 
\begin{align*}
\phi\colon\mcA&\to K\{\tau\}\\
a &\mapsto \phi_a,
\end{align*}
uniquely determined by the image of $T$,
\[
\phi_T = T + g_1 \tau + \cdots+ g_r \tau^r, \hspace{3mm} g_i\in K, \hspace{1mm} g_r\neq 0.
\]
\end{definition}

There is an isomorphism between the (compactified) moduli stack of Drinfeld modules, which we will denote by $\overline{\Dri_r}$, and the weighted projective stack $\mcP(q-1, q^2-1,\dots,q^r-1)$:
\begin{align*}
\overline{\Dri_r}&\xrightarrow{\sim} \mcP(q-1, q^2-1,\dots,q^r-1)\\
 T + g_1 \tau + \cdots+ g_r \tau^r &\mapsto [g_1:\cdots:g_r].
\end{align*} 
In particular, Drinfeld $\F_q[T]$-modules $\phi$ and $\psi$ defined over $\F_q(T)$, characterized by
\[
\phi_T=T + g_1 \tau + \cdots+ g_r \tau^r
\]
and
\[
\psi_T=T + h_1 \tau + \cdots+ h_r \tau^r
\]
respectively, are isomorphic if and only if there exists a $c\in K$ such that  $g_i=c^{q^i-1}h_i$ for all $i$ in $\{1, 2, 3,\dots,r\}$.

Let $\phi$ be a Drinfeld $\mcA$-module over $K$ characterized by
\[
\phi_T=T+g_1\tau+\cdots+g_r\tau^r.
\]
Set $\w_r:=(q-1,q^2-1,\dots,q^r-1)$.

\begin{definition}[Height of a Drinfeld module]\label{def:Drinfeld-height}
 Define the \textbf{height} of the Drinfeld module $\phi$ to be 
\[
\Ht(\phi):= \Ht_{\w_r}([g_1:\dots:g_r])=-\sum_{v\in \Val(K)} \deg(v) \min_{1\leq i\leq r}\left\{\left\lfloor\frac{v(g_i)}{q^i-1}\right\rfloor\right\}.
\]
\end{definition}

\begin{remark}
There are many different ways of defining the height of a Drinfeld module. The height defined above is most similar to the $J$-height and the graded height (see \cite[\S 2.2]{BPR21}). There is also a height defined by Taguchi \cite{Tag93}, which is analogous to the Faltings height. It would be interesting to count Drinfeld modules of bounded height with respect to some of these other heights. To the authors knowledge, the height given in Definition \ref{def:Drinfeld-height}  has not previously appeared in the literature, although it arises naturally from the geometry of the moduli stack of Drinfeld modules.  
\end{remark}

\subsection{Counting Drinfeld modules satisfying finitely many local conditions}

Let $\mcD_r(b)$ denote the set of isomorphism classes of rank $r$ Drinfeld $\mcA$-modules over $K$ of height $b$. Then Theorem \ref{thm:WProjFin} can be used to estimate the cardinality of $\mcD_r(b)$.

\begin{theorem}
The number of rank $r$ Drinfeld $\F_q[T]$-modules over $\F_q(T)$ with height equal to the positive integer $b$ is
\[
\#\mcD_r(b)=\frac{q^{r}}{\zeta_{\F_q(T)}\left(\frac{q^{r+1}-q}{q-1}\right)} q^{b\frac{q^{r+1}-q}{q-1}} +\begin{cases}
 O(q^b b) & \text{ if } r=2,\\
 O\left(q^{b\frac{q^{r+1}-q^2+q-1}{q-1}}\right) & \text{ if } r>2,
\end{cases}
\]
as a function of $b$.
\end{theorem}

\begin{proof}
This result follows from applying Theorem \ref{thm:WProjFin}, with trivial local conditions, to the moduli stack of Drinfeld modules, $\overline{\Dri_r}\cong \mcP(q-1,q^2-1,\dots,q^r-1)$. Note that the number of Drinfeld modules of rank $<r$ (i.e., those with $g_r=0$) is $O\left(q^{b\frac{q^r-q}{q-1}}\right)$, which  is accounted for in the error term.
\end{proof}

\begin{definition}[Minimal equation]
Let $\p\subset \mcA$ be a prime ideal. An equation
\[
\phi_T=T+g_1\tau+\cdots+g_r\tau^r
\] 
defining a Drinfeld $\mcA$-module is said to be \textbf{minimal at $\p$} if $g_j\in \mcA$ for all $j$ and there exists an $i$ for which $v_\p(g_i)<q^i-1$.
\end{definition}

As $\overline{\Dri_r}\cong \mcP(q-1,q^2-1,\dots,q^r-1)$, we see that for each Drinfeld module $\phi$ (up to isomorphism) and prime $\p$, there exists a equation defining $\phi$ which is minimal at $\p$.

We now define several reduction types of Drinfeld modules.

\begin{definition}[Reduction types of Drinfeld modules]\label{def:reduction-types}
Let $\p\subset \mcA$ be a prime ideal. A Drinfeld $\mcA$-module $\phi$ of rank $r$ defined by by a $\p$-minimal equation
\[
\phi_T=T+g_1\tau+\cdots+g_r\tau^r
\]
is said to have \textbf{stable reduction of rank $s$} at $\p$ if $s=\max\{i : \ord_\p(g_i)=0\}$. We simply say that the Drinfeld module $\phi$ has \textbf{stable reduction} at the prime ideal $\p$ if it is stable of any rank; otherwise we say that $\phi$ has \textbf{unstable reduction} at $\p$. We say that the Drinfeld module $\phi$ has \textbf{good reduction} at the prime ideal $\p$ if it has stable reduction of rank $r$; otherwise we say that $\phi$ has \textbf{bad reduction} at $\p$. 
\end{definition}

\begin{theorem}\label{thm:Drinfeld-single-local-condition}
Let $\mcL$ be one of the local conditions in Table \ref{tab:LocalConditions}, and let $\p\subseteq \F_q[T]$ be a prime ideal of norm $N(\p)$. Then the number of isomorphism classes of rank $r$ Drinfeld $\F_q[T]$-modules over $\F_q(T)$ with reduction type $\mcL$ at the prime ideal $\p$, and with height equal to the positive integer $b$, is
\begin{align}\label{eq:single-local-condition}
\kappa_{\mcL}\frac{q^{r}}{\zeta_{\F_q(T)}\left(\frac{q^{r+1}-q}{q-1}\right)} 
 q^{b\frac{q^{r+1}-q}{q-1}} +\begin{cases}
 O(q^b b) & \text{ if } r=2,\\
 O\left(q^{b\frac{q^{r+1}-q^2+q-1}{q-1}}\right) & \text{ if } r>2,
\end{cases}
\end{align}
as a function of $b$,
where $\kappa_{\mcL}=\kappa'_{\mcL}\frac{1}{1-N(\p)^{r-\frac{q^{r+1}-q}{q-1}}}$ with $\kappa'_{\mcL}$ given in Table \ref{tab:LocalConditions}.
\end{theorem}

\begin{proof}
\underline{Stable reduction of rank $s$ case:}
We would like to consider the local condition $\Omega_\p\subset \mcP(q-1,\dots,q^r-1)(K_\p)$ consisting of points corresponding to Drinfeld modules with stable reduction of rank $s$ at $\p$. Let $\F_{\p}$ denote the residue field $\F_q[T]/\p$ at the prime ideal $\p$. The number of rank $r$ Drinfeld modules over $\F_{\p}$ with stable reduction of rank $s$ is
\[
\#\{(g_1,\dots,g_r)\in \F_{\p}^r : g_{s}\not\equiv 0 \pmod{\p},\ g_j\equiv 0 \pmod{\p}\ \forall j>s\}=N(\p)^{s}-N(\p)^{s-1}.
\]
For each element in the above set we may choose a lift of that element $(g_1,\dots,g_r)\in \mcA^r$. Then consider the affine local condition
\[
\Omega_{\p,0}^{\aff}(g_1,\dots,g_r)=\{(x_1,\dots,x_r)\in \mcA_{\p}^r : |x_j-g_j|_\p\leq \frac{1}{N(\p)} \text{ for all } 1\leq j\leq r\},
\]
which has $\p$-adic measure
\[
m_{\p}(\Omega_{\p,0}^{\aff}(g_1,\dots,g_r))=\frac{1}{N(\p)^r}.
\]
Let $\pi_\p$ be a uniformizer at $\p$, and for each non-negative integer $k\in\Z_{\geq 0}$ define the set
\begin{align*}
\Omega_{\p,k}^{\aff}(g_1,\dots,g_r)&=\pi_\p^k\ast_{(q-1,\dots,q^r-1)} \Omega_{\p,0}^{\aff}(g_1,\dots,g_r)\\
&=\{(x_1,\dots,x_r)\in \mcA_\p^r : |x_j-N(\p)^{k(q^j-1)}g_j|_\p \leq \frac{1}{N(\p)^{kq^j-k+1}} \text{ for all } 1\leq j \leq r\}.
\end{align*}
These sets have $\p$-adic measure
\[
m_\p(\Omega_{\p,k}^{\aff})=N(\p)^{-r+k\left(r-\frac{q^{r+1}-q}{q-1}\right)}.
\]
Observe that
\[
\Omega_\p^{\aff}=\bigsqcup_{(g_1,\dots,g_r)}\bigsqcup_{k\geq 0} \Omega_{\p,k}^{\aff}(g_1,\dots,g_r).
\]
Applying Theorem \ref{thm:WProjFin} gives the asymptotic (\ref{eq:single-local-condition}) with
\begin{align*}
\kappa_{\textnormal{stable}, s}
&=m_\p(\Omega_\p^{\aff}\cap\mcA_\p^r)\\
&=\sum_{(g_1,\dots,g_r)}\sum_{k=0}^\infty m_\p(\Omega_{\p,k}^{\aff}(g_1,\dots,g_r)\\
&=(N(\p)^s-N(\p)^{s-1})\sum_{k=0}^{\infty} N(\p)^{-r+k\left(r-\frac{q^{r+1}-q}{q-1}\right)}\\
&=\frac{N(\p)^s-N(\p)^{s-1}}{N(\p)^r}\left(\frac{1}{1-N(\p)^{r-\frac{q^{r+1}-q}{q-1}}}\right).
\end{align*}

\underline{Stable reduction of rank $\geq s$ case:} We sum over the stable of rank $s, s+1,\dots, r$ cases. Doing so we obtain the constant
\[
\kappa'_{\textnormal{stable}, \geq s}=\sum_{i=s}^r \frac{N(\p)^{i}-N(\p)^{i-1}}{N(\p)^r}=\frac{N(\p)^r-N(\p)^{s-1}}{N(\p)^r}.
\]

\underline{Good reduction case:} This is the special case when $s=r$ in the stable of rank $s$ case. We find that
\[
\kappa'_{good}=\frac{N(\p)^{r}-N(\p)^{r-1}}{N(\p)^r}=\frac{N(\p)-1}{N(\p)}.
\]

\underline{Bad reduction case:} This is the complement of the good reduction case, so in this case we find that the constant is
\[
\kappa'_{bad}=1-\kappa'_{good}=\frac{1}{N(\p)}.
\]

\underline{Stable reduction case:} 
This is the special case when $s=1$ in the stable of rank $\geq s$ case. We find that
the constant is
\[
\kappa'_{stable}=\sum_{s=1}^r \frac{N(\p)^{s}-N(\p)^{s-1}}{N(\p)^r}=\frac{N(\p)^r-1}{N(\p)^r}.
\]

\underline{Unstable reduction case:} This is the complement of the stable reduction case, so in this case we find that the constant is
\[
\kappa'_{unstable}=1-\kappa'_{stable}=\frac{1}{N(\p)^r}.
\]
\end{proof}


By Theorem \ref{thm:WProjFin} we can extend Theorem \ref{thm:Drinfeld-single-local-condition} for counting Drinfeld modules satisfying a single local condition to the case of finitely many local conditions, obtaining Theorem \ref{thm:Drinfeld-finite-local-conditions}:  

\begin{reptheorem}[\ref{thm:Drinfeld-finite-local-conditions}]
Let $\p_1,\p_2,\dots,\p_m\subseteq \F_q[T]$ be a finite set of distinct prime ideals of $\F_q[T]$. For each $\p_i$ let $\mcL_i$ be one of the local conditions in Table \ref{tab:LocalConditions}.  Then the number isomorphism classes of rank $r$ Drinfeld $\F_q[T]$-modules over $\F_q(T)$ with reduction type $\mcL_i$ at $\p_i$ for each $i$, and with height equal to the positive integer $b$, is
\[
\frac{q^{r}}{\zeta_{\F_q(T)}\left(\frac{q^{r+1}-q}{q-1}\right)} 
\left(\prod_{i=1}^m \kappa_{\mcL_i}\right) q^{b\frac{q^{r+1}-q}{q-1}} +\begin{cases}
 O(q^b b) & \text{ if } r=2,\\
 O\left(q^{b\frac{q^{r+1}-q^2+q-1}{q-1}}\right) & \text{ if } r>2,
\end{cases}
\]
as a function of $b$.
\end{reptheorem}


\subsection{Counting Drinfeld modules satisfying infinitely many local conditions}


We now address consider the problem of counting Drinfeld modules satisfying and infinite set of local conditions.

\begin{definition}[Everywhere stable reduction]
A Drinfeld module $\phi$ has \textbf{everywhere stable reduction of rank $\geq s$} if it has stable reduction of rank $\geq s$ at all prime ideals of $\mcA$. A Drinfeld module $\phi$ has \textbf{everywhere stable reduction} if it has stable reduction at all prime ideals of $\mcA$ (i.e., everywhere stable reduction of rank $\geq 1$). 
\end{definition}

Recall that $\Val_0(\F_q(T))$ denotes the set of places of $\F_q(T)$ away from the place $\infty=(1/T)$. For each $v\in \Val_0(\F_q(T))$, let $\p_v$ denote the corresponding prime ideal of $\F_q[T]$.

\begin{theorem}\label{thm:Drinfeld-Infinite-Stable}
Let $r$ and $s$ be positive integers with $1\leq s<r$. Then the number of isomorphism classes of rank $r$ Drinfeld $\F_q[T]$-modules over $\F_q(T)$ which are everywhere stable of rank $\geq s$, and have height equal to the positive integer $b$, is asymptotic to
\[
 \frac{\zeta_{\F_q(T)}\left(-r+\frac{q^{r+1}-q}{q-1}\right)}{\zeta_{\F_q(T)}(r-s+1)} \frac{q^{r}}{\zeta_{\F_q(T)}\left(\frac{q^{r+1}-q}{q-1}\right)} q^{b\frac{q^{r+1}-q}{q-1}},
\]
as a function of $b$.
\end{theorem}

\begin{proof}
 We apply our result for counting points on weighted projective stacks satisfying infinitely many local conditions (i.e., Theorem \ref{thm:WProjInfty}), with the local conditions
\[
\Theta_{s,\p} = \{[g_1:\cdots:g_r]\in \mcP(q-1,\dots,q^r-1)(\mcA_\p) : v_\p(g_j)=0 \text{ for some } j\geq s\}.
\]
Note that the local condition $\Theta_{s,\p}$ corresponds to the set of Drinfeld modules with stable reduction of rank $\geq s$ at $\p$.
In order to apply Theorem \ref{thm:WProjInfty} we must check that the local conditions $\Theta_{s,\p}$ satisfy the limit condition (\ref{eq:WProjInfty}) of Theorem \ref{thm:WProjInfty}. We will do this by using Lemma \ref{lem:WPROJ-Ek-Bha}.

First note that $\Theta_{r-1,\p} \subseteq \Theta_{s,\p}$ for all $1\leq s\leq r-1$. It therefore suffices to show that the local conditions $\Theta_{r-1,\p}$ satisfy the limit criterion (\ref{eq:WProjInfty}). Consider the weighted projective stack $\mcP_\mcA(q-1,\dots,q^r-1)$, with coordinates $[x_1:\cdots:x_r]$. Let $Y$ denote the closed subscheme defined by $x_{r-1}=x_r=0$. Then the local conditions defined in Lemma \ref{lem:WPROJ-Ek-Bha},
\[
\Omega_\p=\{x\in \mcP(\w)(\mcA) : x\pmod{\p} \not\in Y(\F_\p)\},
\]
coincide with the local conditions $\Theta_{r-1,\p}$ for all prime ideals $\p$. In particular, by Lemma \ref{lem:WPROJ-Ek-Bha}, the local conditions $\Theta_{r-1,\p}$ satisfy the required limit condition (\ref{eq:WProjInfty}).

As the local conditions $\Theta_{s,\p}$ satisfy the limit criterion (\ref{eq:WProjInfty}), we may apply Theorem \ref{thm:WProjInfty}. Doing this, and using the constants from Table \ref{tab:LocalConditions} for the proportion of Drinfeld modules with stable reduction of rank $\geq s$, we have that the number of rank $r$ Drinfeld $\F_q[T]$-modules over $\F_q(T)$ with everywhere stable reduction of rank $\geq s$, and  height equal to the positive integer $b$, is asymptotic to 
\begin{align*}
&\left(\prod_{v\in \Val_0(\F_q(T))} \frac{N(\p_v)^r-N(\p_v)^{s-1}}{N(\p_v)^r}\cdot \frac{1}{1-N(\p)^{r-\frac{q^{r+1}-q}{q-1}}}\right)  \frac{q^{r}}{\zeta_{\F_q(T)}\left(\frac{q^{r+1}-q}{q-1}\right)} q^{b\frac{q^{r+1}-q}{q-1}}\\
 &\hspace{1cm}=\frac{\zeta_{\F_q(T)}\left(-r+\frac{q^{r+1}-q}{q-1}\right)}{\zeta_{\F_q(T)}(r-s+1)} \frac{q^{r}}{\zeta_{\F_q(T)}\left(\frac{q^{r+1}-q}{q-1}\right)} q^{b\frac{q^{r+1}-q}{q-1}}\\
 &\hspace{1cm}=\left(1-q^{-(r-s)}\right)\left(1-q^{-(r-s+1)}\right)\zeta_{\F_q(T)}\left(-r+\frac{q^{r+1}-q}{q-1}\right) \frac{q^{r}}{\zeta_{\F_q(T)}\left(\frac{q^{r+1}-q}{q-1}\right)} q^{b\frac{q^{r+1}-q}{q-1}},
\end{align*}
as a function of $b$.
\end{proof}




Note that the proof of Theorem \ref{thm:Drinfeld-Infinite-Stable} still works if we change finitely many of the local conditions. Doing so we obtain Theorem \ref{thm:Drinfeld-Infinite-Local-Conditions}:

\begin{reptheorem}[\ref{thm:Drinfeld-Infinite-Local-Conditions}]
Let $r$ and $s$ be positive integers with $1\leq s<r$. Let $S=\{\p_1,\dots,\p_m\}$ be a finite set of distinct prime ideals of $\F_q[T]$. For each $\p_i$ let $\mcL_i$ be one of the local conditions in Table \ref{tab:LocalConditions}. Then the number of isomorphism classes of rank $r$ Drinfeld $\F_q[T]$-modules over $\F_q(T)$ which have stable reduction of rank $\geq m$ everywhere outside of $S$, and reduction type $\mcL_i$ at $\p_i$ for each $\p_i\in S$, and of height equal to the positive integer $b$, is asymptotic to
\[
 \zeta_{\F_q(T)}(r-s+1)^{-1}\left(\prod_{\p_i\in S} \frac{\kappa'_{\mcL_i}}{1-N(\p_i)^{-r+s-1}}\right) \frac{q^{r}}{\zeta_{\F_q(T)}\left(\frac{q^{r+1}-q}{q-1}\right)} q^{b\frac{q^{r+1}-q}{q-1}},
\]
as a function of $b$.
\end{reptheorem}


\bibliographystyle{plain}

\end{document}